\documentclass[12pt]{article}%

\usepackage{amsmath,enumerate}
\usepackage{amsfonts}
\usepackage{amssymb}
\usepackage{color}

\newtheorem{theorem}{Theorem}[section]

\newtheorem{definition}[theorem]{Definition}

\newtheorem{lemma}[theorem]{Lemma}

\newtheorem{problem}[theorem]{Problem}
\newtheorem{proposition}[theorem]{Proposition}

\newtheorem{remark}[theorem]{Remark}

\newenvironment{proof}[1][Proof]{\noindent\textbf{#1.} }
{\hfill \ \rule{0.5em}{0.5em}}

\newcommand{\turanf}{\textup{ex}(n,F)}

\newcommand\blfootnote[1]{%
  \begingroup
  \renewcommand\thefootnote{}\footnote{#1}%
  \endgroup
}

\begin{document}

\title{\vspace{-1cm}Orthogonal polarity graphs and Sidon sets}
\author{Michael Tait\thanks{Department of
Mathematics, University of California, San Diego, CA, USA, mtait@math.ucsd.edu} \and Craig Timmons\thanks{Department of
Mathematics, University of California, San Diego, CA, USA, ctimmons@math.ucsd.edu} }

\maketitle
\blfootnote{\textsuperscript{1}Both authors were supported by NSF Grant DMS-1101489 through Jacques Verstra\" ete.}

\vspace{-1cm}

\begin{abstract}
Determining the maximum number of edges in an $n$-vertex $C_4$-free graph is a well-studied problem that dates back to a paper of Erd\H{o}s from 1938.  One of the most important families of $C_4$-free graphs are the
Erd\H{o}s-R\'{e}nyi orthogonal polarity graphs.  We show that the Cayley sum graph constructed using a Bose-Chowla Sidon set is isomorphic to a large induced subgraph of the Erd\H{o}s-R\'{e}nyi orthogonal polarity graph.  Using this isomorphism we prove that the Petersen graph is a subgraph of every sufficiently large Erd\H{o}s-R\'{e}nyi orthogonal polarity graph.  
\end{abstract}

\section{Introduction}

Let $\mathcal{F}$ be a family of graphs.  A graph $G$ is \emph{$\mathcal{F}$-free} if $G$ does not contain a subgraph that is isomorphic to a graph in $\mathcal{F}$.  The \emph{Tur\'{a}n number} of $\mathcal{F}$, denoted
$\textup{ex}(n , \mathcal{F} )$, is the maximum number of edges in an $n$-vertex graph that is $\mathcal{F}$-free.
When $\mathcal{F} = \{ F \}$, we write $\turanf$ instead of $\textup{ex}(n, \{ F \})$.
Determining Tur\'{a}n numbers for different families of graphs is one of the most studied problems in extremal graph theory.  A case of particular interest is when $\mathcal{F} =\{ C_4\}$.  A counting argument of
K\H{o}v\'{a}ri, S\'{o}s, and Tur\'{a}n \cite{kst} (see also Reiman \cite{reiman}) gives
\begin{equation}\label{eq:kst ub}
\textup{ex}( n , C_4) \leq \frac{1}{2} n^{3/2} + \frac{1}{2}n.
\end{equation}
In 1966, Brown \cite{b}, Erd\H{o}s, R\'{e}nyi, and S\'{o}s \cite{ers} independently constructed graphs that show (\ref{eq:kst ub}) is asymptotically best possible.  These graphs are often called Erd\H{o}s-R\'{e}nyi polarity graphs and are constructed using an orthogonal polarity of the projective plane $PG(2,q)$.  The construction is as follows.  Let $q$ be a prime power.
The \emph{Erd\H{o}s-R\'{e}nyi graph}, denoted $ER_q$, is the graph whose vertices are the points of
$PG(2, q)$.  Two distinct vertices $(x_0 , x_1 , x_2)$ and $(y_0 , y_1 , y_2)$ are adjacent if
$x_0 y_0 + x_1 y_1 + x_2 y_2 = 0$.  It is well known that these graphs have $q^2  + q  + 1$ vertices,
have $\frac{1}{2}q (q+1)^2$ edges, and are $C_4$-free.  Thus for any prime power $q$,
\begin{equation}\label{eq:er lb}
\textup{ex}(q^2 + q + 1 , C_4) \geq \frac{1}{2}q(q + 1)^2.
\end{equation}
F\"{u}redi \cite{fu1983, fu1996} proved that (\ref{eq:er lb}) is best possible and moreover, any
$C_4$-free graph with $q^2 + q + 1$ vertices and $\frac{1}{2}q (q  +1)^2$ edges is
an orthogonal polarity graph of a projective plane of order $q$,  provided $q \geq 15$.  
The graphs $ER_q$ play a central role in the Tur\'{a}n problem for $C_4$.  A difficult and related unsolved conjecture of McCuaig is that every $C_4$-free $n$-vertex graph with $\textup{ex}(n , C_4)$ edges is a subgraph of some polarity graph (see \cite{fu1994}).

Now we introduce another family of $C_4$-free graphs that have many edges.
Let $q$ be a prime power and let $\theta$ be a generator of the multiplicative group $\mathbb{F}_{q^2}^{*}$, the nonzero elements of the finite field $\mathbb{F}_{q^2}$.  The set
\[
A (q , \theta ) := \{ a \in \mathbb{Z}_{q^2 - 1} : \theta^{a} - \theta \in \mathbb{F}_{q} \}
\]
is a \emph{Bose-Chowla Sidon set}.  These sets, constructed in \cite{bc}, have $q$ elements and are Sidon sets in $\mathbb{Z}_{q^2-1}$. That is, they have the property that the sums
$a + b ( \textup{mod}~q^2 -1 )$ with $a,b \in A(q, \theta)$ are all distinct, up to an ordering of the terms.

\begin{definition}\label{def:Gq}
Let $q$ be a prime power and $\theta$ be a generator of $\mathbb{F}_{q^2}^*$.
The graph $G_{q, \theta }$ is the graph with vertex set $\mathbb{Z}_{q^2 - 1}$.  Two distinct vertices $x$ and $y$ are adjacent if and only if $x + y  \in A(q, \theta )$.
\end{definition}

A vertex in $G_{q, \theta}$ has degree $q$ or $q-1$ and a vertex of degree $q-1$ is called an \emph{absolute point}.
In \cite{tt} the graphs $G_{q , \theta}$ for odd $q$ were used to improve a lower bound of Abreu et al.\ \cite{abl} on the Tur\'{a}n number $\textup{ex}( q^2 - q - 2 , C_4)$, $q$ an odd prime power. The eigenvalues of $G_{q,\theta}$ were studied in \cite{chung} and in a more general setting in \cite{v}. In this paper we continue the study of the graphs
$G_{q, \theta}$.  Our main result shows that $G_{q, \theta}$ is an induced subgraph of $ER_q$.

\begin{theorem}\label{th:induced subgraph}
Let $q$ be a prime power and $\theta$ be a generator of $\mathbb{F}_{q^2}^*$.
The graph $G_{q, \theta}$ is isomorphic to an induced subgraph of the Erd\H{o}s-R\'{e}nyi graph $ER_q$.
\end{theorem}

In the course of proving Theorem \ref{th:induced subgraph} we explicitly describe how one can obtain $ER_q$ from
$G_{q , \theta}$.  This gives a new method in which one can construct $ER_q$. 

The vertices at distance two from the absolute points in $ER_5$ induce a graph that is isomorphic to the Petersen graph \cite{coxeter}.  
As an application of Theorem \ref{th:induced subgraph}, we prove that the Petersen graph is contained in every $ER_q$ provided $q \geq 3$.  

\begin{theorem}\label{th:petersen in erq}
For any prime power $q \geq 3$, the Petersen graph is a subgraph of $ER_q$.
\end{theorem}

In \cite{tt} it is shown that $G_{q, \theta}$ has $\frac{1}{2}q^3 - q^2 - O( q^{3/4})$ edges when $q$ is an odd prime power.  
Using some elementary properties of $A(q, \theta)$, we can determine the number of edges of $G_{q, \theta}$ exactly.  

\begin{proposition}\label{cor:odd q construction}
If $q$ is an odd prime power and $\theta$ is a generator of $\mathbb{F}_{q^2}^*$, then the graph $G_{q, \theta}$ is a $C_4$-free graph with $q^2 - 1$ vertices, and $\frac{1}{2}q^3 - q + \frac{1}{2}$ edges.
\end{proposition}

A natural question is if the graphs of Corollary \ref{cor:odd q construction} are extremal.  Using $ER_q$, we can show that this is not the case when $q$ is a square prime power.

\begin{proposition}\label{prop:odd q again}
Let $p$ be a prime and $t \geq 1$ be an integer.  If $q = p^{2t}$, then
\[
\textup{ex}( q^2 - 1 , C_4) \geq \frac{1}{2}q^3 + \frac{1}{2} q^{3/2} - \frac{3}{2} q + \frac{1}{2} q^{1/2} - 2.
\]
\end{proposition}

When $q$ is a power of 2, it is easy to determine the number of edges of $G_{q, \theta}$.
In this case, $G_{q, \theta}$ has exactly $q$ absolute points.  Furthermore, there is a unique vertex $x$ whose neighborhood is precisely the set of absolute points of $G_{q , \theta}$ (see Section 3).  Let $H_{q, \theta}$ be the subgraph of $G_{q, \theta}$ obtained by removing this vertex $x$ and its neighbors which are the vertices of degree $q-1$.
The graph $H_{q, \theta}$ implies
$\textup{ex}(q^2 - q -2 , C_4) \geq \frac{1}{2}q^3 -  q^2$ , $q$ a power of 2.  This lower bound was first obtained by
Abreu, Balbuena, and Labbate \cite{abl} using a different method.

\begin{proposition}\label{prop:power of 2}
Let $q$ be a power of 2 and $\theta$ be a generator of $\mathbb{F}_{q^2}^*$.  The
graph $G_{q, \theta}$ has $\frac{1}{2}q^3 - q$ edges, and contains an induced subgraph
$H_{q, \theta}$ that has $q^2 - q - 2$ vertices, and has
$\frac{1}{2}q^3 - q^2$ edges.
\end{proposition}

It is still an open problem to determine if the lower bound
\[
\textup{ex}(q^2  - q - 2 , C_4) \geq \frac{1}{2}q^3 - q^2, ~ q ~\mbox{a power of 2}
\]
is best possible.  It has been conjectured \cite{abl} that it is indeed optimal.
Another result that we prove is that $G_{q, \theta}$ has diameter 3.  In fact, we will prove something stronger.
Given a graph $F$, we say that a graph $G$ is \emph{$F$-saturated} if $G$ is $F$-free and adding any edge to $G$ creates a copy of $F$.

\begin{theorem}\label{th:saturation}
If $q$ is a prime power and $\theta$ is a generator of $\mathbb{F}_{q^2}^*$, then the graph
$G_{q, \theta }$ is $C_4$-saturated.  When $q$ is a power of 2 and $q>4$, the graph $H_{q, \theta}$ is
$C_4$-saturated.
\end{theorem}

A recent result of Firke, Kosek, Nash, and Williford \cite{fknw} is that
\[
\textup{ex}(q^2 + q , C_4) \leq \frac{1}{2}q(q+1)^2  - q
\]
for even $q$.  By deleting a vertex of degree $q$ from $ER_q$, we find
\[
\textup{ex}( q^2 + q , C_4) = \frac{1}{2}q(q + 1)^2 - q
\]
whenever $q$ is a power of 2.  It is also announced in \cite{fknw} that when $q$ is a power of 2, the $C_4$-free graphs with $q^2 + q$ vertices and $\frac{1}{2}q ( q+ 1)^2 - q$ edges must be those obtained by deleting a vertex of degree $q$ from 
an orthogonal polarity graph.  This leads to the following question: under what conditions is a subgraph of $ER_q$ extremal?
One can use Theorem \ref{th:saturation} and Proposition \ref{prop:odd q again} to deduce that there are $n$-vertex subgraphs of $ER_q$ that are $C_4$-saturated but have less than $\textup{ex}(n , C_4)$ edges.  This shows that $C_4$-saturation is not enough to force a subgraph of $ER_q$ to be extremal.

The rest of this paper is organized as follows.  In Section 2 we establish some basic properties about $A(q, \theta)$ and prove both  Propositions \ref{cor:odd q construction} and \ref{prop:odd q again}.  In Section 3 we prove Proposition \ref{prop:power of 2} and in Section 4, we prove Theorem \ref{th:saturation}.  Our main theorem,
Theorem \ref{th:induced subgraph}, is proved in Section 5.  Using Theorem \ref{th:induced subgraph} we prove Theorem
\ref{th:petersen in erq} in Section 6.


\section{Preliminary results and proofs Propositions 1.4 and 1.5}

In this section $q$ is a prime power, $\theta$ is a generator of  $\mathbb{F}_{q^2}^*$, and $A = A ( q, \theta)$ is a Bose-Chowla Sidon set.  Here and throughout the rest of the paper, 
\[
H = \{ 0 , q + 1 , \dots , (q -2)(q + 1) \}
\]
is the subgroup of $\mathbb{Z}_{q^2 - 1}$ generated by $q + 1$.  The first lemma is known.  We write 
$A - A$ for the set $\{a - b : a , b \in A \}$.  

\begin{lemma}\label{le:not q+1}
If $A$ is a Bose-Chowla Sidon set in $\mathbb{Z}_{q^2 - 1}$ and $H = \langle q + 1 \rangle$, then 
\begin{center}
$A \cap H = \emptyset$  and  $A - A = \mathbb{Z}_{q^2 - 1} \backslash \{ q + 1 , 2(q + 1) , \dots , (q -2)(q +1) \}$.
\end{center}
\end{lemma}
\begin{proof}
If $t (q + 1) \in A$, then for some $b \in \mathbb{F}_q$, $\theta^{ t ( q + 1) } = \theta + b$.
Since $\theta^{ t ( q + 1)} \in \mathbb{F}_q$, we have $\theta = \theta^{ t ( q + 1) } - b \in \mathbb{F}_q$ which is a contradiction.  

Let $a_i , a_j \in A$ and suppose that $a_i - a_j = t (q + 1)$, $\theta^{a_i} = \theta + b_i$, and $\theta^{a_j} = \theta + b_j$ where $b_i , b_j \in \mathbb{F}_q$.  Then 
\[
\theta + b_i = ( \theta + b_j ) \theta^{t ( q + 1) }.
\]
Equating coefficients of $\theta$, we obtain $\theta^{t ( q + 1)  } = 1$ so that $t (q + 1) = 0$ in $\mathbb{Z}_{q^2- 1}$.  
This shows that $(A - A) \cap \{ q + 1 , 2(q + 1) , \dots , (q - 2)(q + 1) \} = \emptyset$.  Since there are 
$q ( q - 1) = q^2 - q$ nonzero elements of $A  - A$ as $A$ is a Sidon set, we must have
$A - A = \mathbb{Z}_{q^2 - 1} \backslash \{ q + 1 , 2(q + 1) , \dots , (q -2)(q +1) \}$.   
\end{proof}

\begin{lemma}\label{new lemma 1}
If $A$ is a Bose-Chowla Sidon set in $\mathbb{Z}_{q^2 - 1}$, then we can label the elements of $A$ so that   
$A = \{ a_1 , \dots , a_q \}$ and  
\[
a_i = i + m_i ( q + 1)
\]
with $m_i \in \{ 0 , 1, \dots , q- 2 \}$ for $1 \leq i \leq q$.  
\end{lemma}
\begin{proof}
By Lemma \ref{le:not q+1}, no difference of two distinct elements of $A$ are congruent to $0$ modulo $q + 1$.  
\end{proof}

\bigskip

\begin{proof}[Proof of Proposition \ref{cor:odd q construction}]
The graphs $G_{q, \theta}$ are known to be $C_4$-free.  It is shown in \cite{tt} that if
$A (q , \theta )$ has $t$ even elements, then the number of absolute points in $G_{q, \theta}$ is $2t$.
Since $q$ is odd, we have $t = \frac{1}{2} (q - 1)$ by Lemma \ref{new lemma 1}.  To see this, write 
$A = \{a_1 , \dots , a_q \}$ where $a_i = i + m_i ( q + 1)$.  Then $a_i$ is even when viewed as an element of $\mathbb{Z}$ precisely when $i$ is even.  The number of vertices of $G_{q, \theta}$ is $q^2 - 1$.  The number of edges of
$G_{q , \theta}$ is
\[
e(G_{q , \theta}) = \frac{1}{2} ( q (q^2 - 1 - 2t ) + (q - 1)(2t) ) = \frac{1}{2}q^3 - \frac{1}{2} q - t = \frac{1}{2}q^3 - q + \frac{1}{2}.
\]
\end{proof}

\bigskip

\begin{proof}[Proof of Proposition \ref{prop:odd q again}]
Let $p$ be a prime and $t \in \mathbb{N}$.  Let $q = p^{2t}$.
The field $\mathbb{F}_q$ contains $\mathbb{F}_{ q^{1/2} }$ as a subfield.  By considering the vertices 
in $ER_q$ whose coordinates can be taken to be in $\mathbb{F}_{q^{1/2}}$, we see that $ER_{q^{1/2}}$ is a subgraph of $ER_q$.  Let $H$ be a copy of $ER_{q^{1/2}}$ in $ER_q$.  By Brooks' Theorem,
\[
\alpha (H) \geq \frac{ n(H) }{ \Delta (H)} = \frac{q + q^{1/2} + 1}{q^{1/2} + 1} = q^{1/2} + \frac{1}{q^{1/2}+1}.
\]
Choose an independent set $I \subset V(H)$ with $|I| = q^{1/2} - 1$ and let $S = V(H) \backslash I$.  Then
$|S| = q + 2$, $e(S) = \frac{1}{2}q^{1/2} ( q^{1/2 } + 1 )^2$, and so
\begin{eqnarray*}
e(ER_q \backslash S ) & \geq & \frac{1}{2}q(q+1)^2 - (q + 2)(q+ 1) + \frac{1}{2}q^{1/2} ( q^{1/2} + 1)^2 \\
& = & \frac{1}{2}q^3 + \frac{1}{2}q^{3/2} - \frac{3}{2}q + \frac{1}{2} q^{1/2}  - 2.
\end{eqnarray*}
This shows that for any square prime power $q$,
\[
\textup{ex}(q^2 - 1 , C_4) \geq  \frac{1}{2}q^3 + \frac{1}{2}q^{3/2} - \frac{3}{2}q + \frac{1}{2} q^{1/2}  - 2.
\]
\end{proof}


\section{Proof of Proposition 1.6}

In this section $q$ is a power of 2, $\theta$ is a generator of $\mathbb{F}_{q^2}^*$, and $A = A(q, \theta)$ is a Bose-Chowla Sidon set.

\begin{lemma}\label{le:absolute neighborhood}
If $q=2^{ \alpha}$ and $\alpha \in \mathbb{N}$, then there is an $x \in \mathbb{Z}_{q^2 - 1}$ such that
\[
x + y \in A
\]
for all $y \in \mathbb{Z}_{q^2 - 1}$ which satisfy $y + y \in A$.
\end{lemma}
\begin{proof}
Suppose $y + y \in A$.  Then $2y \equiv a ( \textup{mod}~2^{2 \alpha} - 1 )$ for some $a \in A$ which implies
$y \equiv 2^{ 2 \alpha -1} a (\textup{mod}~2^{2 \alpha} - 1)$.  If $A = \{a_1 , \dots , a_q \}$, then the absolute points are
\[
2^{2 \alpha - 1} a_1 , 2^{2 \alpha - 1} a_2 , \dots , 2^{ 2 \alpha -1} a_q.
\]
A vertex $x$ is adjacent to each of these vertices if
\begin{equation}\label{eq:special x}
x+ 2^{2 \alpha - 1} \cdot A = A
\end{equation}
where
\[
k \cdot A := \{ k a : a \in A \}.
\]
To show that there is such an $x$, we will use the following result of
Lindstr\"{o}m \cite{l}.

\begin{theorem}[Lindstr\"{o}m]\label{th:lindstrom 1}
Let $p$ be a prime, $q = p^m$, $m \in \mathbb{N}$, and $\theta$ be a generator of $\mathbb{F}_{q^2}^*$.
Let $B(q, \theta) = \{ b \in \mathbb{Z}_{q^2 - 1} : \theta^{bq} + \theta^q = 1 \}$.  If
$c \in \mathbb{Z}_{q^2-1}$ and $\theta^c + \theta = \theta^q$, then
\begin{equation*}\label{eq:multiply by 2}
A(q , \theta)   - c \equiv B(q , \theta)  (\textup{mod}~q^2 - 1 ).
\end{equation*}
Furthermore, $p \cdot B(q, \theta) \equiv B( q , \theta) ( \textup{mod}~q^2 - 1)$.
\end{theorem}

We apply Theorem \ref{th:lindstrom 1} with $p =2$ and $q = 2^{ \alpha}$ to get
\[
2^{2 \alpha - 1} \cdot A - 2^{2 \alpha - 1} c \equiv 2^{2 \alpha - 1} \cdot B(q, \theta) \equiv
B(q, \theta) \equiv A - c ( \textup{mod}~q^2- 1 ).
\]
If we take $ x= c - 2^{2 \alpha - 1}c $, then (\ref{eq:special x}) holds and we are done.
\end{proof}

\bigskip

\begin{proof}[Proof of Proposition \ref{prop:power of 2}]
Consider $G_{q, \theta}$ where $q = 2^{ \alpha}$, $\alpha \in \mathbb{N}$.  Let $H_{q,\theta}$ be the subgraph of $G$ obtained by
removing the vertex $x$ of Lemma \ref{le:absolute neighborhood} and all of the neighbors of $x$.  Then $H_{q,\theta}$ has
$q^2 - 1 - (q + 1) = q^2 - q - 2$ vertices, and
\[
e(H_{q,\theta}) = \frac{1}{2} ( (q^2 - 1)q - q ) - q - q(q - 2) = \frac{1}{2}q^3 - q^2.
\]
\end{proof}


\section{Proof of Theorem 1.7}

In this section $q$ is a prime power, $\theta$ is a generator of  $\mathbb{F}_{q^2}^*$, and $A = A ( q, \theta)$ is a Bose-Chowla Sidon set.

\begin{lemma}\label{le:paths}
For any $x \in \mathbb{Z}_{q^2 - 1} \backslash A$, there are at least $q-1$ ordered triples
$(a,b,c) \in A^3$ with $a  - b + c = x$, $a \neq b$, and $b \neq c$.
\end{lemma}
\begin{proof}
We consider two cases.

\noindent
\textit{Case 1:} $x = j ( q + 1)$ with $1 \leq j \leq q - 1$.

Let $a \in A$ and consider $x - a$.  By Lemma \ref{le:not q+1}, $a$ is not of the form $t (q + 1)$ for any $1 \leq t \leq q - 1$.
This implies that $x - a$ is also not of the form $t(q+1)$ with $1 \leq t \leq q- 1$.  
By Lemma \ref{le:not q+1}, there is a pair $b,c \in A$ with $c-b = x - a$.  Since $x - a \neq 0$, the pair $b,c$ is unique because $A$ is a Sidon set.
We can rewrite this equation as $a - b + c = x$.  If $a = b$ or $c = b$, then $x \in A$ which is impossible by
Lemma \ref{le:not q+1}.  There are $q$ choices for $a$.  Also, there is at most one $a \in A$ for which
$a - b +a = x$.  To see this, suppose we also have $a' - b + a' = x$ for some $a' \in A$.  Then
$a + a  = a'  + a'$ and by the Sidon property, $a = a'$.  This shows that there is at most one triple $(a,b,c)$ for which
$a = c$, $a - b + c = x$, and $a \neq b$, $c \neq b$.  We conclude that in this case, there are at least
$q - 1$ desired triples.

\noindent
\textit{Case 2:} $x \neq j (q +1)$ for any $1 \leq j \leq q - 1$.

Suppose $A = \{a_1 , \dots , a_q \}$ and consider the $q$ distinct elements
\[
x - a_1 , x - a_2, \dots , x - a_q.
\]
We will show that at most one of these elements is of the form
$j(q + 1)$ with $1 \leq j \leq q - 1$.  Suppose that
$x - a_i = j(q +1)$ and $x - a_{i'} = j' (q + 1)$ where $i \neq i'$ and $1 \leq j \neq j' \leq q -1$.  Then
\[
a_i - a_{i'} = (x - j ( q+1) ) - (x - j' (q + 1) ) = (j' - j )( q +1)
\]
which contradicts Lemma \ref{le:not q+1}.  By relabeling if necessary, we may assume that
\[
x - a_i \notin \{ q+1 , 2(q + 1) , \dots , (q - 2)(q + 1) \}
\]
for $1 \leq i \leq q - 1$.
Let $i \in \{1,2, \dots , q- 1 \}$ and consider $x - a_i$.  There is a unique pair
$b_i , c_i \in A$ with $x - a_i = c_i -b_i$.  We can rewrite this equation as $x = a_i - b_i + c_i$.  If $a_i \neq b_i$ and $b_i \neq c_i$, then we are done.  Assume that $a_i = b_i$ or $b_i = c_i$.  Then $x \in A$, which contradicts our hypothesis.  Finally note that
each of the triples $(a_1, b_1, c_1), \dots , (a_{q-1} , b_{q-1}, c_{q-1})$ are all distinct since
$a_1, \dots , a_{q-1}$ are all distinct.
\end{proof}

\bigskip

\begin{proof}[Proof of Theorem \ref{th:saturation}]
Using Lemma \ref{le:paths}, we can now prove that $G_{q, \theta}$ is $C_4$-saturated.
Let $x,y \in \mathbb{Z}_{q^2 - 1}$ and suppose that $x$ and $y$ are not adjacent in $G_{q , \theta}$ so
$x + y \notin A$.  By Lemma \ref{le:paths}, there is a triple $a,b,c \in A$ with
$x + y = a - b + c$ where $a \neq b$ and $b \neq c$.  Let $z = a-x$ and $t = c-y$.  By construction,
$x$ is adjacent to $z$ and $y$ is adjacent to $t$.  Since $z + t  = x + y -a - c = b$, we also have that
$z$ is adjacent to $t$.  In order to conclude that $xzty$ is a path of length 3, we must argue that $x,z,t$, and $y$ are all distinct.

If $z = x$, then $a =0$ but $0 \notin A$.  Therefore, $z \neq x$ and similarly $t \neq y$.  Now $x$ and $y$ are not adjacent so that $z \neq y$, and $t \neq x$.  The last possibility is if $z = t$.  If this occurs, then $z + z = z + t = b$ and so $z$ is an absolute point.  In the case when $q$ is odd, we know that any vertex is adjacent to at most two absolute points (see \cite{tt}).  In the case when $q$ is even, there is only one vertex that is adjacent to more than one absolute point by Lemma \ref{le:absolute neighborhood}. Therefore there are at most $2$ triples that yield $z$ an absolute point.  By Lemma \ref{le:paths}, there are at least $q-3$ triples that can be chosen so that $z$ is not an absolute point.

The conclusion is that whenever $x + y \notin A$, there are at least $q-3$ paths of length 3 between $x$ and $y$ and so adding the edge $xy$ gives a 4-cycle.  This completes the proof of Theorem \ref{th:saturation} in the case of $G_{q, \theta}$.

We will use the following remark, that we have proved something stronger than $C_4$ saturation, in the proof of the case of
$H_{q , \theta}$.

\begin{remark}\label{re:number of c4s}
If $q$ is any prime power and $H$ is a graph obtained by adding an edge to $G_{q,\theta}$, then $H$ contains at least $q-3$ copies of $C_4$.
\end{remark}

It remains only to show that $H_{q,\theta}$ is $C_4$ saturated. Let $uv \not\in E(H_{q,\theta})$ and let $H = H_{q,\theta} \cup \{uv\}$ and $G = G_{q,\theta} \cup \{uv\}$.  By Remark \ref{re:number of c4s}, $G$ contains at least $q-3$ copies of $C_4$, all of which have $uv$ as an edge.  Since $H_{q , \theta}$ is obtained by removing the vertex $x$ and the set of absolute points from $G_{q, \theta}$, it suffices to show that this process destroys less than $q-3$ copies of $C_4$.

If removing $x$ and its neighborhood destroys a copy of $C_4$ that has $uv$ as an edge, then that $C_4$ must have as a vertex either $x$ or one of its neighbors. Since the entire neighborhood of $x$ is removed, neither $u$ nor $v$ is adjacent to $x$, and so the $C_4$ cannot have $x$ as a vertex.  Therefore the $C_4$ has $y$ as a vertex for some absolute point $y$. However, note that because $G_{q,\theta}$ is $C_4$-free and because $x$ is adjacent to every absolute point, any vertex not equal to $x$ can have at most one neighbor that is an absolute point. In particular, $u$ and $v$ are each connected to at most one absolute point, meaning that removing $x$ and its neighborhood can destroy at most $2 < q-3$ copies of $C_4$.
\end{proof}


\section{Proof of Theorem 1.2}

In this section $q$ is a power of a prime, $\theta$ is a generator of $\mathbb{F}_{q^2}^*$, and $A=A(q, \theta)$ is a
Bose-Chowla Sidon set.  The subgroup of $\mathbb{Z}_{q^2 - 1}$ generated by $q+1$ is denoted by $H$.  

\begin{lemma}\label{le:distances}
If $x$ and $y$ are distinct vertices of $G_{q, \theta}$, then the following hold.

\noindent
(i) $d(x,y) = 1$ if and only if $x + y \in A$.

\noindent
(ii) $d(x,y) = 2$ if and only if $x + y \notin A$ and $x-y \notin H$.

\noindent
(iii) $d(x,y) = 3$ if and only if $x+y \notin A$ and $x - y \in H$.
\end{lemma}
\begin{proof}
The proof of (i) is the definition of adjacency in $G_{q, \theta}$.

Suppose $d(x,y) = 2$.  Certainly $x+y \notin A$ otherwise $x$ and $y$ are adjacent.  Let $z$ be a common neighbor of $x$ and $y$.  There are elements $a , b \in A$ such that $x + z = a$ and $z + y = b$.  Thus
$x - y = a - b$ and $a \neq b$ since $x \neq y$.  By Lemma \ref{le:not q+1},
\[
x - y \in ( A - A ) \backslash \{ 0 \} =  \mathbb{Z}_{q^2 - 1} \backslash \{ q + 1 , 2(q+1) , \dots , (q - 2) ( q + 1) \}.
\]
Convserely, suppose $x+y \notin A$ and $x - y \notin H$.  By Lemma \ref{le:not q+1}, we can write
$x - y = a - b$ for some $a,b \in A$ with $a \neq b$.  Let $z = a - x$.  Then $x + z \in A$ and
$z + y \in A$ so that $z$ is a common neighbor of $x$ and $y$.  Furthermore, $z \neq x$ and $z \neq y$ since $x$ and $y$ are not adjacent.

The case when $d(x,y) = 3$ now follows from (ii), and the fact that $G_{q , \theta}$ has diameter at most 3 which is a consequence of Theorem \ref{th:saturation}.
\end{proof}

\begin{lemma}\label{le:H lemma}
If $x$ and $y$ are vertices in $H$, then $d(x,y) \in \{ 0 , 3 \}$.
\end{lemma}
\begin{proof}
If $x$ and $y$ are in $H$, then $x + y \in H$ and $x - y \in H$.  By Lemma \ref{le:not q+1}, $x + y \notin A$ and we are done by
Lemma \ref{le:distances}.
\end{proof}

\bigskip

Given $1 \leq i \leq q$, let $H + i = \{ h + i : h \in H \}$.

\begin{lemma}\label{le:Hi lemma}
Let $0 \leq i \leq q$.  If $x$ and $y$ are distinct vertices in $H + i$, then $x$ and $y$ are not joined by a path of length two.
\end{lemma}
\begin{proof}
Let $x,y \in H + i$, say $x = h + i$ and $y = h' + i$ where $h , h ' \in H$.  Then $x - y = h - h' \in H \backslash \{ 0 \}$.
In the proof of Lemma \ref{le:distances}, it is shown that if two distinct vertices $z$ and $t$ are joined by a path of
length two, then $z - t \in (A - A ) \backslash \{ 0 \}$.  Since $x - y \in H \backslash \{ 0 \}$, we have by
Lemma \ref{le:not q+1} that $x - y \notin A - A$.
\end{proof}

\bigskip

\begin{proof}[Proof of Theorem \ref{th:induced subgraph}]
We are going to show that $G_{q, \theta}$ is a subgraph of $ER_q$ by adding a a set of $q + 2$ vertices and
some edges to $G_{q, \theta}$ to obtain a graph that is $C_4$-free, has $q^2 + q  +1$ vertices, and has
$\frac{1}{2}q(q + 1)^2$ edges.  We will then give an isomorphism from this graph to $ER_q$.  It will be convenient to add vertices and edges to $G_{q, \theta}$ in steps.  At each step we will check that the new graph is $C_4$-free.

\noindent
\textit{Step 1:} Constructing $G_1$.

Add vertices $z_0 , z_1 , \dots , z_q$ to $G_{q, \theta}$.  For $0 \leq i \leq q$, make $z_i$ adjacent to every vertex in the coset $H+i$.  Let this graph be $G_1$.  The graph $G_1$ has $q^2 - 1 + q +1 = q^2 + q$ vertices, and has
$e(G_{q, \theta}) + (q + 1)(q  - 1) = e(G_{q, \theta} ) + q^2 - 1$ edges.  The cosets
$H , H+1 , \dots , H + q$ are all pairwise disjoint so that there is no $C_4$ that contains two
vertices from the set $\{z_0 , z_1 , \dots , z_q \}$.  Suppose that there is a $C_4$ of the form $z_i (h + i) t (h' + i)$ for some
$0 \leq i \leq q$, $t \in V(G_{q, \theta})$, and $h,h' \in H$.  This implies $h+i$ and $h'+i$ are joined by a path of length two in $G_{q, \theta}$ contradicting Lemma \ref{le:Hi lemma}.  

\noindent
\textit{Step 2:} Constructing $G_2$.

Add a vertex $y$ to $G_1$ and make $y$ adjacent to $z_0, z_1, \dots , z_q$.  Call this graph $G_2$.  The graph
$G_2$ has $q^2 + q + 1$ vertices and $e(G_{q, \theta}) + q^2 + q$ edges.  Any $C_4$ in $G_2$ must be of the form
$y z_i t z_j$ for some $t \in V(G_{q,\theta})$ and $0 \leq i \neq j \leq q$.  This is impossible however as the cosets
$H+i$ and $H+j$ are disjoint.

\noindent
\textit{Step 3:} Constructing $G_3$.

For $1 \leq j \leq \lfloor \frac{q}{2} \rfloor$, make $z_j$ adjacent to $z_{q+1-j}$.  Call this graph $G_3$.
The number of vertices of $G_3$ is $q^2 + q + 1$, and the number of
edges of $G_3$ is
\[
e(G_3) = \left\{
\begin{array}{ll}
e( G_{q, \theta} ) + q^2 + q + \frac{q-1}{2} & \mbox{if $q$ is odd}, \\
e(G_{q , \theta}) + q^2 + q + \frac{q}{2}  & \mbox{if $q$ is even}.
\end{array}
\right.
\]
A potential $C_4$ in $G_3$ must use one of the edges $z_j z_{q + 1 - j}$.  Since $z_0$ is not adjacent to any of
$z_1, \dots , z_q$ and $y$ is the unique common neighbor of $z_j$ and $z_{q+1 - j}$, there is no $C_4$ that contains
$z_j$, $z_{q+1 - j}$ and $y$.  The only remaining possibility is a $C_4$ of the form
$z_j z_{q+1 - j} (h + q + 1 -j)(h' + j)$ where $h , h ' \in H$.  This is impossible as
$( h + q + 1 -  j  ) + (h ' + j) = h + h' + q + 1 \in H$ and $A \cap H = \emptyset$ by Lemma \ref{le:not q+1}.  

When $q$ is odd, $e(G_{q, \theta}) = \frac{1}{2}q^3 - q + \frac{1}{2}$ by Corollary \ref{cor:odd q construction}.  This gives
\[
e(G_3) = e(G_{q, \theta}) + q^2 + q + \frac{q-1}{2} = \frac{1}{2}q (q + 1)^2.
\]
When $q$ is even, $e(G_{q, \theta}) = \frac{1}{2}q^3 - q$ by Proposition \ref{prop:power of 2} and so $e(G_3) = \frac{1}{2}q ( q + 1)^2$.  In either case, the graph $G_3$ is $C_4$-free with $q^2 + q + 1$ vertices and $\frac{1}{2}q (q + 1)^2$ edges.
By a result of F\"{u}redi \cite{fu1996}, $G_3$ is an orthogonal polarity graph of a projective plane of order $q$.  
To prove $G_3$ is isomorphic to $ER_q$, we will give an isomorphism.    

We will use an alternative definition of adjacency in $ER_q$ which yields a graph isomorphic to the graph obtained using the original definition (see \cite{mw}).  Vertices $(x_0, x_1, x_2)$ and $(y_0, y_1, y_2)$ in $ER_q$ are adjacent if and only if
\begin{equation}\label{eq:ERq}
x_0y_2 + x_2y_0 = x_1y_1.
\end{equation}

The graph $G_3$ has $q^2 + q + 1$ vertices and $q^2 - q$ of them are of the form $a_i + t ( q +1)$ for some $i \in [q]$ and $0 \leq t \leq q- 2$.  Finding an algebraic relation akin to (\ref{eq:ERq}) that tells us when $a_i + t_1 ( q+1)$ is adjacent to 
$a_j + t_2 ( q + 1)$ is the main obstacle.  
Let $A  = \{a_1 , \dots , a_q \}$ and suppose $\theta^{a_i} = \theta + b_i$ where $\mathbb{F}_q = \{b_1 , \dots , b_q \}$.  
The key observation is that 
$a_i + t_1 ( q + 1) \in A  + t_1 ( q + 1) $ is adjacent to $a_j + t(q + 1) \in A + t_2 ( q + 1)$ if and only if 
\[
b_i + b_j = \mu^{t_1 + t_2} - \alpha
\]
where the primitive root $\theta$ satisfies $\theta^2 = \alpha \theta  + \beta$ and $\mu = \theta^{ - (q + 1)}$.  
To see this, suppose that $a_i + t_1 ( q + 1)$ is adjacent to $a_j + t_2 ( q + 1)$, i.e. $(a_i + t_1 ( q + 1) ) + ( a_j + t_2(q+1) ) \in A$.  Then 
\[
\theta^{ a_i + t_1 ( q+ 1) } \theta^{ a_j + t_2(q+1)} = \theta^a 
\]
for some $a \in A$ so that
\[
\theta^2 + \theta ( b_i + b_j )  + b_i b_j = \theta^{ - (t_1 + t_2) ( q + 1) } ( \theta + b) = \mu^t ( \theta + b)
\]
where $\theta^a = \theta + b$.  Since $\theta^2 = \alpha \theta + \beta$, we can equate coefficients of $\theta$ to get 
$\alpha + b_i + b_j = \mu^{t_1+ t_2}$.  Summarizing, we have $a_i + t_1 ( q+ 1)$
is adjacent to $a_j + t_2 ( q + 1)$ if and only if 
\[
 b_i + b_j = \mu^{t_1 + t_2} - \alpha
\]   
which can be viewed as an analog of (\ref{eq:ERq}).  Our isomorphism depends on the parity of $q$ and we deal with the odd case first.   

\textbf{Case 1}: $q$ is an odd prime power.  

Let $q$ be an odd prime power.  Assume that the primitive root $\theta$ satisfies $\theta^2 = \alpha \theta  +\beta$ where $\alpha$, $\beta \in \mathbb{F}_q$.  Assume $A$ has been labeled so that $a_i  = i + m_i ( q + 1)$ where 
$A = \{a_1 , \dots , a_q \}$.    
Let $\mu = \theta^{ - (q + 1) }$ and $\delta = 2^{ - 1} ( \alpha - 1)$.  Define the map $\phi: V (G_3) \to V(ER_q)$ as follows.

\begin{enumerate}
\item For $1 \leq i \leq q$ and $0 \leq t \leq q - 2$, let 
\[
\phi ( a_i + t ( q+1) ) = ( 1 , \mu^{t } - 1 , - \mu^{t} + \alpha + b_i - \delta ).
\]
\item For $0 \leq t \leq q - 2$, let $\phi ( ( q - 1 - t)(q + 1) ) = ( 0 , 1, \mu^{  t } - 1 )$.
\item For $1 \leq i \leq q$, let $\phi ( z_i) = ( 1 , - 1 , - b_i  - \delta )$, $\phi (z_0 ) = (0,0,1)$, and 
$\phi ( y) = (0 , 1, -1)$.  
\end{enumerate}

Let $1 \leq i \leq q$ and let $0 \leq t_1 \leq q - 2$.  

First consider $a_i + t_1 ( q+ 1)$.  In $ER_q$, the neighbors of 
\[
\phi ( a_i + t_1  ( q + 1))  = (1 , \mu^{t_1} - 1 , - \mu^{t_1} + \alpha + b_i - \delta )
\]
are 
$ (0,1, \mu^{t_1} - 1 ) = \phi ( ( q - 1 - t_1 )(q + 1) )$, $(1 , - 1 , b_i - \delta ) = \phi (z_i)$, and all vertices in the set
\begin{equation}\label{eq:1.1}
 \{ 
(1 , \mu^{t_2} - 1 , ( \mu^{t_2} - 1)( \mu^{t_1}  - 1) + \mu^{t_1} - \alpha - b_i + \delta ) : 0 \leq t_2 \leq q - 2 \}.
\end{equation}
A straightforward computation shows that (\ref{eq:1.1}) is the same as the set
\begin{equation*}
\{ \phi ( a_j + t_2 ( q + 1) ) : b_i + b_j = \mu^{t_1+ t_2} - \alpha \}.
\end{equation*}
In $G_3$, the neighbors of $a_i + t_1 ( q + 1)$ are $(q - 1 - t_1)(q +1)$, $z_i$, and 
all $a_j + t_2 ( q + 1)$ for which $b_i + b_j = \mu^{t_1 + t_2}  - \alpha$.  We conclude that 
$a_i + t_1 ( q +1)$ is adjacent to $x$ in $G_3$ if and only if 
$\phi (a_i + t_1 ( q + 1) )$ is adjacent to $\phi (x)$ in $ER_q$.  
  
\smallskip

Next we consider $(q - 1 - t_1 )(q + 1)$.  In $ER_q$, the neighbors of 
\[
\phi ( ( q - 1  - t_1) ( q + 1 ) ) = ( 0 , 1 , \mu^{t_1} - 1 )
\]
are $(0,0,1)= \phi ( z_0 )$ and $\{ ( 1 , \mu^{t_1} - 1 , z ) : z \in \mathbb{F}_q \} = \phi ( A + t_1 ( q + 1) )$.  In $G_3$, the neighbors of $(q - 1 - t_1)(q + 1)$ are $z_0$ and $A + t_1 ( q + 1) $.  This shows that 
$(q - 1 - t_1)(q + 1)$ is adjacent to $x$ in $G_3$ if and only if 
$\phi ( ( q- 1 - t_1)(q + 1) )$ is adjacent to $\phi (x)$ in $ER_q$.

Now consider $z_i$ where $0 \leq i \leq q$.  By definition, $\phi (z_0) = (0,0,1)$ and the neighbors of this vertex are 
$(0,1,-1) = \phi (y)$ and $\{ (0,1 , z) : z \in \mathbb{F}_q \backslash \{-1 \} \} = \phi (H)$.  In $G_3$, the neighbors of $z_0$ are 
$y$ and all vertices in $H$.  When $1 \leq i \leq q$, the neighbors of 
$\phi (z_i) = (1,-1 , -b_i - \delta )$ are $(0,1,-1) = \phi (y)$, $( 1 , \mu^t - 1 , - \mu^t + \alpha + b_i - \delta ) = \phi (a_i + t ( q + 1))$ for $0 \leq t \leq q - 2$, and $(1,-1, - b_{q + 1 - i} - \delta ) = \phi (z_{q + 1 - i } )$ where the subscript 
$q + 1 - i$ is taken modulo $q$.  To verify that 
$(1 , -1 , -b_{q + 1 - i } - \delta )$ is indeed a neighbor of $(1 , -1 , -b_i - \delta)$, we need to have
\[
- b_{q + 1 - i } - \delta - b_i -  \delta = 1
\]
since $\delta = 2^{-1} ( \alpha - 1)$.  This equation is equivalent to $b_i + b_{q + 1 - i} + \alpha = 0$.  This holds as 
\[
a_i + a_{q + 1 - i } = (q + 1) + ( m_i + m_{q + 1 - i } )( q + 1)
\]
and so $( \theta  + b_i ) ( \theta + b_{q + 1 - i } ) \in \mathbb{F}_q$ which means $\alpha + b_i + b_{q+1 - i } = 0$.
In $G_3$, the neighbors of $z_i$ are $y$, $z_{ q+ 1 - i}$, and all vertices in $H+i$ which is the same as the set 
$\{ a_i  + t ( q + 1) : 0 \leq t \leq q - 2 \}$.    
 
The last vertex to check is $y$ but this is not necessary as all edges incident to $y$ have an endpoint in 
$V(G_3) \backslash \{ y \}$.    

\medskip

\textbf{Case 2}: $q$ a power of 2.

Let $q$ be a power of 2 and $\theta$ be a primitive root of $\mathbb{F}_{q^2}$ chosen so that 
$\theta^2 = \theta + \beta$ for some $\beta \in \mathbb{F}_q$.  Such a $\theta$ exists by a result of 
Moreno \cite{moreno}.  We use the same notation as in Case 1.
Define the map $\phi$ as follows.  

\begin{enumerate}
\item For all $1 \leq i \leq q$ and $0 \leq t \leq q - 2$, let
\[
\phi ( a_i + t ( q +1 ) ) = ( 1 , \mu^t + 1 , b_i + 1 + \mu^t ).
\]
\item For $0 \leq t \leq q - 2$, let $\phi ( (q - 1 - t )(q + 1) ) = ( 0 , 1 , \mu^t + 1 )$.
\item For $1 \leq i \leq q$, let $\phi (z_i) = (1,1, b_i)$, $\phi (z_0) = (0,0,1)$, and $\phi (y) = (0,1,1)$.  
\end{enumerate}

We will be a little more brief in this case as the ideas are quite similar to the $q$ odd case. 

Let $1 \leq i \leq q$ and $0 \leq t_1 \leq q - 2$ and consider $a_i  + t_1 ( q + 1)$.  The neighbors of
\[
\phi (a_ i + t_1 ( q  + 1) ) = ( 1 , \mu^{t_1} - 1 , - \mu^{t_1}+ \alpha + b_i - \delta )
\]
are $(0 , 1, \mu^{t_1} + 1) = \phi ( ( q - 1 - t_1) ( q + 1) )$, $(1,1,b_i) = \phi (z_i)$, and 
\begin{eqnarray*}
(1 , \mu^{t_2}  + 1 , ( \mu^{t_2} + 1)( \mu^{t_1} + 1) + b_i + \mu^{t_1 + t_2} + \mu^{t_1} ) 
& = & (1 , \mu^{t_2} + 1 , b_j + 1 + \mu^{t_2} ) \\ &=& \phi (a_j + t_2 ( q + 1) )
\end{eqnarray*}
provided $b_i + b_j = \mu^{t_1 + t_2} + 1$.  In $G_3$, the neighbors of 
$a_i + t_1 ( q+1)$ are $(q - 1 - t_1) ( q + 1)$, $z_i$, and $a_j + t_2 ( q + 1)$ for any pair $j,t_2$ for which 
$b_i + b_j = \mu^{t_1 + t_2} + 1$.  

Next we consider $(q - 1 - t_1)(q+ 1)$.  In this case, the neighbors of 
$\phi ( ( q - 1 - t_1 ) ( q+ 1) )$ are 
$(0,0,1)  = \phi (z_0)$ and all vertices in the set 
\[
\{ (1 , \mu^{t_1} + 1 , z ) : z \in \mathbb{F}_q \} = \phi ( A + t_1 ( q + 1) ).
\]
The neighbors of $(q - 1 - t_1)(q + 1)$ in $G_3$ are $z_0$ and all vertices in $A + t_1 ( q + 1)$.

Now consider $z_i$ where $1 \leq i \leq q$.  The neighbors of $\phi (z_i) = (1,1,b_i)$ are 
$(0,1,1) = \phi (y)$, all vertices in the set 
\[
\{ (1 , \mu^t + 1 , \mu^t + 1 + b_i ) : 0 \leq t \leq q - 2 \} = \phi ( \{ a_i + t ( q+ 1) : 0 \leq t \leq q - 2 \} ),
\]
and $(1,1, 1 + b_i) = \phi ( z_{q + 1 - i })$.  To see this last equation, we must show that 
$1 + b_i = b_{q + 1 - i }$ where again these subscripts are reduced modulo $q$.  Since 
\[
a_i + a_{q + 1 - i } = i + m_i ( q + 1) + (q + 1 - i ) + m_{q + 1 - i }( q + 1) = ( 1 + m_i + m_{q+1 - i } )(q + 1),
\]
we have that $( \theta + b_i )( \theta + b_{q + 1 - i } ) \in \mathbb{F}_q$ and so $1 + b_i + b_{q + 1 - i } = 0$.  
The neighbors of $z_i$ in $G_3$ are $y$, $z_{ q+1 - i}$, and all vertices in the set 
$H + i$ which is equal to $\{ a_i + t ( q _ 1) : 0 \leq t \leq q - 2 \}$.  

The only vertices that remain are $z_0$ and $y$ but we leave these easy cases to the reader.

\medskip

In either case, $\phi$ is an isomorphism from $G_3$ to $ER_q$ and $G_{q , \theta}$ is isomorphic to an induced subgraph of $ER_q$.
\end{proof}


\section{Proof of Theorem 1.3}

In this section $q \geq 4$ is a prime power, $\theta$ is a generator of $\mathbb{F}_{q^2}^*$, and 
$A = A(q, \theta)$ is a Bose-Chowla Sidon set.  Write 
\[
a_1 = 1 + m_1 (q+1) , a_2 = 2 + m_2 (q+1) , \dots , a_q = q + m_q (q+1)
\]
where $\{a_1 , \dots , a_q \} = A$ and $m_i \in \{0,1, \dots , q \}$ for each $i$.  
Let $1 \leq i < j < k \leq q$ be distinct integers with $i + j - k = 0$.  Let $s$ and $t$ be integers with 
$0 \leq s \neq t \leq q  -2$ and 
\[
s + t \equiv - ( m_i + m_j - m_k) ( \textup{mod}~q-1).
\]
We have 
\begin{eqnarray*}
a_i + a_j - a_k & \equiv & i + j - k + (m_i + m_j - m_k)(q +1) ( \textup{mod}~q^2 - 1 ) \\
& \equiv & -(s + t)(q + 1) ( \textup{mod}~q^2 - 1).
\end{eqnarray*}
This implies 
\begin{equation}\label{eq:new eq3}
a_i + a_j + (s + t)(q + 1) \in A.
\end{equation}
Define the following six vertices in $G_{q, \theta}$:
\begin{center}
$u_1 = a_i + s ( q+ 1)$, ~~~ $v_1 = a_i + t ( q + 1)$, ~~~ $w_1 = ( q - 1 - s)(q + 1)$

$u_2 = a_j + s( q + 1)$, ~~~ $v_2 = a_j + t ( q + 1)$, ~~~ $w_2 = (q - 1 - t)(q + 1)$.
\end{center}
Once can check that all of these vertices are distinct using the fact that $i \neq j$, $0 \leq s \neq t \leq q  - 2$, 
$A \cap H = \emptyset$, and $(A - A) \cap H  = \{ 0 \}$ (recall $H = \{0,q + 1 , \dots , (q - 2)(q + 1) \}$.  
By (\ref{eq:new eq3}), $u_1$ is adjacent to $v_2$ and $u_2$ is adjacent to $v_1$.  
Since $u_1 + w_1$, $u_2 + w_2$, $v_1 + w_1$, and $v_2 + w_2$ are all in $A$, we have that 
$u_1 w_1 u_2 v_1 w_2 v_2$ is a 6-cycle in $G_{q, \theta}$.  

Now we use the vertices $z_0 , z_1 , \dots , z_q$, and $y$ that are added to $G_{q, \theta}$ to obtain $ER_q$ (see 
\textit{Steps} 1-3 in the proof of Theorem \ref{th:induced subgraph}).  In $ER_q$, vertex $z_0$ is adjacent to both $w_1$ and $w_2$, vertex $z_i $ is adjacent to both $u_1$ and $v_1$, and vertex $z_{j}$ is adjacent to both $u_2$ and $v_2$.  The final vertex that we need is $y$ which is adjacent to $z_0$, $z_{i }$, and  $z_{j}$.  This shows that the set of vertices 
\[
\{ u_1 , u_2 , v_1 , v_2 , w_1 , w_2 , z_0 , z_i , z_j , y \}
\]
contain a Petersen graph in $ER_q$.


\section{Concluding Remarks}

In \cite{mw}, Mubayi and Williford investigated the independence number of $ER_q$.  An open problem mentioned in
\cite{mw} is to construct an independent set $I$ in $ER_q$ for any $q$ that is not an even power of 2 with
$|I| = q^{3/2} + O(q)$, or show that no such set exists.
Our main result, Theorem \ref{th:induced subgraph} shows that $G_{q, \theta}$ is an induced subgraph of $ER_q$.  These two graphs differ by only $q + 2$ vertices and so the independence number of $G_{q, \theta}$ differs from the independence number of $ER_q$ by at most $q+2$.  Using the definition of adjacency, we find that an independent set $I$ in $G_{q, \theta}$ is a set $I \subset \mathbb{Z}_{q^2 - 1}$ such that $x + y \notin A(q, \theta)$ for all $x \neq y$, $x,y \in I$; that is
\[
\alpha (G_{q, \theta}) = \max_{ I \subset \mathbb{Z}_{q^2 - 1} }
\{ |I| : (  (I + I ) \backslash (2 \cdot I) ) \cap A(q, \theta ) = \emptyset \}
\]

By computing the second eigenvalue of $ER_q$, one can show that $\alpha (ER_q) \leq q^{3/2} + O(q)$.  Improvements in the error term have been made (see \cite{mw, gn}) but all of these improvements involve eigenvalues in some way.  It was pointed out to the authors by Javier Cilleruelo \cite{jc personal} that one can use the main result of \cite{jc} to show that
\begin{equation}\label{eq:independence number}
\alpha (G_{q, \theta}) \leq q^{3/2} + O(q).
\end{equation}
One does not need to compute any of the eigenvalues of $G_{q, \theta}$ to obtain (\ref{eq:independence number}) in this manner.  This gives a new proof of the estimate $\alpha (ER_q) \leq q^{3/2} + O(q)$ that does not use eigenvalues.

Determining the maximum size of an independent set in $ER_q$ is equivalent to the following question, which may be of independent interest:

\begin{problem}
What is the maximum number of pairwise non-orthogonal $1$-dimensional subspaces of $\mathbb{F}_q^3$?
\end{problem}

Another open problem from \cite{mw} is to find an induced subgraph of $ER_q$, $q$ a power of 2, that is triangle free and has at least $\frac{q^2}{2} + O(q^{3/2})$ vertices.  Again, since this problem concerns induced subgraphs, finding such a subgraph in $G_{q, \theta}$ would suffice to solve this problem in $ER_q$.  Thus, it is of interest to find the largest set
$J \subset \mathbb{Z}_{q^2 - 1}$ such that for any $x,y,z \in J$ with $x,y,z$ all distinct, at least one of the sums
$x+ y$, $x  +z$, or $y + z$ is not contained in $A(q,\theta )$.

\section{Acknowledgment}

The authors would like to thank Jason Williford for finding an important error in an earlier version of this paper.  


\end{document}